\documentclass[12pt]{amsart}
\usepackage{amssymb}
\usepackage{hyperref}
\usepackage{graphicx}
\usepackage{subfigure}

\newtheorem{theorem}{Theorem}[section]

\newtheorem{proposition}[theorem]{Proposition}

\newtheorem{lemma}[theorem]{Lemma}

\newcommand{\st}{\  : \ }

\newcommand{\e}{\varepsilon}

\newcommand{\R}{\mathbf{R}}
\newcommand{\C}{\mathbf{C}}
\newcommand{\cH}{\mathcal{H}}
\newcommand{\cR}{\mathcal{R}}

\newcommand{\D}{\mathcal{D}}

\newcommand{\iy}{\infty}
\DeclareMathOperator{\E}{\mathbf{E}}
\DeclareMathOperator{\Var}{\mathbf{Var}}

\DeclareMathOperator{\tr}{Tr}
\renewcommand{\P}{\mathbf{P}}
\newcommand{\Id}{\mathrm{Id}}

\newcommand{\ketbra}[2]{| #1 \rangle \langle #2 |}

\newcommand{\ket}[1]{| #1 \rangle}

\renewcommand{\leq}{\leqslant}
\renewcommand{\geq}{\geqslant}

\begin{document}

\title{Realigning random states}

\author{Guillaume Aubrun}
\address{G.A.: Institut Camille Jordan, Universit\'e Claude Bernard Lyon 1, 43 boulevard du 11 novembre 1918, 69622 Villeurbanne CEDEX, France {\tt aubrun@math.univ-lyon1.fr}}

\author{Ion Nechita}
\address{I.N.: CNRS, Laboratoire de Physique Th\'eorique, IRSAMC, Universit\'e de Toulouse, UPS, 31062 Toulouse, France. {\tt nechita@irsamc.ups-tlse.fr}}


\begin{abstract}

We study how the realignment criterion (also called computable cross-norm criterion) succeeds asymptotically in detecting whether random states are separable or entangled. We consider random states on $\C^d \otimes \C^d$ obtained by partial tracing a Haar-distributed random pure state on $\C^d \otimes \C^d \otimes \C^s$ over an ancilla space $\C^s$. We show that, for large $d$, the realignment criterion typically detects entanglement if and only if $s \leq (8/3\pi)^2 d^2$. In this sense, the realignment criterion is asymptotically weaker than the partial transposition criterion.

\end{abstract}

\maketitle

\section*{Introduction}

A central problem in Quantum Information Theory is to decide whether a state is separable or entangled. Although this is a computationally hard task \cite{gur}, several operational criteria have been proposed to detect entanglement, such as the Peres partial transpose criterion (PPT) \cite{peres} and the realignment criterion due to Rudolph and Chen--Wu \cite{cwu,rud}.

In this paper, we focus on the realignment criterion, study its behaviour on large dimensional bipartite systems and compute the threshold for induced random states. These are random states on $\cH=\C^d \otimes \C^d$ which are obtained as the partial trace over an ancilla space $\C^s$ of a Haar-distributed random pure state on $\cH \otimes \C^s$. We show that a threshold for the realignment criterion occurs at $s_0=(8/3\pi)^2 d^2\approx 0.72d^2$, in the following sense: if the ancilla dimension $s$ is smaller that $s_0$, with large probability the realignment criterion detects that the random state is entangled; when $s$ is larger than $s_0$, the realignment criterion fails. Since the threshold for the PPT criterion is larger (it occurs at $s=4d^2$, see \cite{aub}), this means that the realignment criterion is asymptotically weaker that the Peres criterion.

Here is a more striking illustration: take a random state uniformly distributed (with respect to the Lebesgue measure) on the convex body of all mixed states on $\C^d \otimes \C^d$ (this corresponds to $s=d^2$, see \cite{zs}). Then, with probability tending to $1$ as $d$ tends to infinity, the following occurs: this state is entangled, and its entanglement is detected by the Peres criterion but not by the realignment criterion.

Our proofs are based on a new model in Random Matrix Theory: realigned Wishart matrices. Let $W$ be a Wishart random matrix ($W=XX^*$, where $X$ is a $d^2 \times s$ matrix with i.i.d. Gaussian entries). We study the realignment of the difference $W-s\Id$ (which is a non-Hermitian matrix), and show that the asymptotic singular value distribution is given, under proper normalization, by a quarter circle law.

The fact that the realignment criterion is generically asymptotically weaker that the Peres criterion is illustrated in a qualitative fashion when we focus on unbalanced bipartite systems. We consider induced random states on $\C^{d_1} \otimes \C^{d_2}$ (the dimension of the environment being still $s$). When $d_1$ is fixed and $d_2$ tends to infinity, we show that the threshold for the realignment criterion is  $s=d_1^2$. This is to be compared with the corresponding threshold for the Peres criterion, which is not bounded with respect to $d_2$ (it was shown in \cite{bn1} that this threshold is $s=\alpha(d_1)d_2$, with $\alpha(d_1)=2d_1+2\sqrt{d_1^2-1}$). 

The paper is organized as follows: Section \ref{sec:intro} introduces the background and states the main theorem about
the threshold for the realignment criterion on $\C^d \otimes \C^d$ (Theorem \ref{thm:threshold-realignment}). 
Section \ref{sec:wishart} introduces realigned Wishart matrices, and contains Theorem \ref{theo:quartercircle} about convergence to the quartercircle distribution. Section \ref{sec:from-wishart-to-states} contains a derivation of Theorem \ref{thm:threshold-realignment} from Theorem \ref{theo:quartercircle}. Section \ref{sec:graphical-calculus} introduces the graphical calculus which is used in the proof of Theorem \ref{theo:quartercircle}. In section \ref{sec:moments-wishart} we compute the moments of realigned Wishart matrices, and the proof of Theorem \ref{theo:quartercircle} is completed in Section \ref{sec:proof-moments}. Finally, Section \ref{sec:unbalanced} deals with unbalanced tensor products.

\section{Background and statement of the results}

\label{sec:intro}

\subsection{Permutation criteria}

We consider a bipartite Hilbert space $\cH=\C^{d_1} \otimes \C^{d_2}$. Let $\{e_i\}_{i=1}^{d_1}$ and $\{f_j\}_{j=1}^{d_2}$ be the canonical bases of $\C^{d_1}$ and $\C^{d_2}$. Any operator $A$ on $\cH$ admits a (double-indexed) matrix representation 
\[ A = \sum_{i,k}^{d_1} \sum_{j,l}^{d_2} A_{ij,kl} \ketbra{e_i \otimes f_j}{e_k \otimes f_l} .\]

For every permutation $\sigma$ of the indices $\{i,j,k,l\}$, we can introduce the corresponding reshuffling operation, which maps the matrix $A = (A_{ij,kl})$ to the (possibly non-square) matrix $A^\sigma = (A_{\sigma(i)\sigma(j),\sigma(k)\sigma(l)})$. This operation depends on the particular choices of bases.

If $\rho$ is a pure product state, then $\|\rho^\sigma\|_1 = 1$. Consequently, for any separable state $\rho$, we have 
$\|\rho^\sigma\|_1 \leq 1$. Each permutation yields to an operational separability criterion. Such reshufflings were studied in \cite{hhh}, where it was proved that each of these 24 reshuffling is equivalent to one of the following
\begin{enumerate}
 \item The trivial reshuffling, where $\|A^\sigma\|_1=\|A\|_1$ for every $A$.
 \item The partial transposition \cite{peres}, denoted $A^\Gamma$, which corresponds to swapping the indices $j$ and $l$. This operation is equivalently described as $A^\Gamma = (\Id \otimes T)A$, where $T$ is the usual transposition of $d \times d$ matrices. Note that for a state $\rho$, the condition $\|\rho^\Gamma\|_1 \leq 1$ is equivalent to $\rho^\Gamma \geq 0$, and this is known as the PPT criterion (positive partial transpose).
 \item The realignment \cite{cwu,rud}, denoted $A^R$, which corresponds to swapping the indices $j$ and $k$. We have
 \[ A^R =  \sum_{i,k}^{d_1} \sum_{j,l}^{d_2} A_{ij,kl} \ketbra{e_i \otimes e_k}{f_j \otimes f_l} .\]
The resulting matrix $A^R$ has dimension $d_1^2 \times d_2^2$. The fact that a separable state $\rho$ must satisfy $\|\rho^R\|_1 \leq 1$ is called the {\itshape realignment criterion} or the {\itshape computable cross-norm criterion}.
\end{enumerate}

Except in Section \ref{sec:unbalanced}, we focus on the balanced case ($d_1=d_2=d$). In this case the matrix $A^R$ is square; however note that the realignment of an Hermitian matrix does not produce a Hermitian matrix in general. Note also that $\Id^R = d E$, where $E=E_d$ is the maximally entangled state $E=\ketbra{\psi}{\psi}$, with $\ket{\psi} = \frac{1}{\sqrt{d}} \sum_{i=1}^d \ket{e_i \otimes e_i}$. Obviously, one has also $\Id=(dE)^R$.

It is known \cite{rud} that when $d \geq 3$, neither of the PPT or realignment criteria is stronger (one can find states which violate one criterion and satisfy the other one). In this paper, we show that when $d$ is large, the 
PPT criterion is generically stronger that the realignment criterion.

\subsection{Random states}

We consider the standard model of induced random states. These mixed states are obtained as partial traces (over some environment) of Haar-distributed random pure states. More precisely, we denote by $\mu_{n,s}$ the distribution of the state
\[ \tr_{\C^s} \ketbra{\psi}{\psi} ,\]
where $\psi$ is uniformly distributed on the unit sphere in $\C^n \otimes \C^s$. In the following, we identify $\C^n$ with $\C^d \otimes \C^d$, for $n=d^2$.

When $s \geq n$, the probability measure $\mu_{n,s}$ has a density with respect
to the Lebesgue measure on the set of states on $\C^n$ which has a simple form \cite{zs}
\begin{equation} \label{eq:formula-density} \frac{d\mu_{n,s}}{d \mathrm{vol}}(\rho)
= \frac{1}{Z_{n,s}} (\det \rho)^{s-n} ,\end{equation}
where $Z_{n,s}$ is a normalization factor. Note that formula \eqref{eq:formula-density}
allows to define the measure $\mu_{n,s}$ for every real $s \geq n$, while the partial trace construction makes sense only for integer values of $s$.

The dimension $s$ of the environment can be thought of as a parameter. The resulting mixed state is more likely to be entangled when $s$ is small. On the other hand, if $s \to \iy$, the resulting mixed state converges to the maximally mixed state, which is separable. Therefore, for any separability criterion, we expect a threshold phenomenon---a critical value between the range of $s$ where the criterion is generically true and the range where it is generically false. Known results in this direction include, for a random state $\rho \in \C^d \otimes \C^d$ with distribution $\mu_{d^2,s}$,

\begin{enumerate}
 \item For separability vs entanglement, the threshold is of order $s \approx d^3$. More precisely, $\rho$ is typically entangled when $s \precsim d^3$, and typically  separable when $s \succsim d^3 \log^2 d$ \cite{asy}.
 \item For the PPT criterion, the threshold occurs precisely at $s=4d^2$:  $\rho$ is typically not-PPT when $s \leq 4d^2$ and typically PPT when  $s \geq 4d^2$ \cite{aub}.
\end{enumerate}
(in this context, a property is called typical if the probability that it holds goes to $1$ as $d$ tends to infinity).

We show that the threshold for the realignment criterion is precisely at $s=(8/3\pi)^2d^2$. This is our main theorem. 

\begin{theorem} \label{thm:threshold-realignment}
Denote $\gamma=(8/3\pi)^2 \approx 0.72$. For every $\e>0$, there exist positive constants $c(\e),C(\e)$ such that the following holds. If $\rho$ is a random state on $\C^d \otimes \C^d$ with distribution $\mu_{d^2,s}$, then
\begin{enumerate}
 \item If $s\leq (\gamma-\e)d^2$, then 
 \[ \P( \| \rho^R \|_1 >  1) \geq 1-C(\e) \exp ( -c(\e) \max(s,d^{1/4}) ). \] 
 \item If $s \geq (\gamma+\e)d^2$, then
 \[ \P( \| \rho^R \|_1 \leq 1) \geq 1-C(\e) \exp ( -c(\e)s). \] 
\end{enumerate}
\end{theorem}

A comparison between the thresholds for the Peres and realignment criteria shows that the latter is {\it generically asymptotically weaker}. When the environment dimension $s$ is between $(8/3\pi)^2d^2$ and $4d^2$, random states are non-PPT, but the realignment criterion fails to detect entanglement. As noted in the introduction, this range includes the special case $s=d^2$, which corresponds to the uniform measure on the set of states (the density in equation \eqref{eq:formula-density} being constant).

The theorem will follow from the description of the limiting distribution of singular values of $\rho^R$, which are
shown to converge towards a quarter-circle distribution. This result can be equivalently stated using Wishart matrices instead of random quantum states, which are more convenient from a random matrix theory perspective. We develop this approach in the following section.

We also consider the case of unbalanced states, i.e. $\C^{d_1} \otimes \C^{d_2}$ with $d_1$ fixed, and $d_2$ tending to infinity. In this asymptotic regime, we show that the threshold for the realignment criterion is exactly $s=d_1^2$. Since the threshold occurs for a finite value of $s$, the realignment criterion is \emph{qualitatively weaker} in the unbalanced case than the Peres criterion (see \cite{bn1}).

\section{Spectral distribution of realigned Wishart matrices}

\label{sec:wishart}

We describe here a new result from random matrix theory, which is the main ingredient in the proof of Theorem \ref{thm:threshold-realignment}. Let $X$ be a $d^2 \times s$ random matrix with i.i.d. $\mathcal N_\C(0,1)$ (standard complex Gaussian) entries, and $W=XX^\dagger$. The random matrix $W$ is known as a Wishart random matrix of parameters $(d^2,s)$. We consider $W$ as an operator on $\C^d \otimes \C^d$, therefore we can consider the realignment of $W$, denoted $R=W^R$. We may write $W_{d^2,s}, R_{d^2,s}, \dots$ instead of $W,R,\dots$ if we want to make dimensions explicit. 

It turns out that to be simpler to study the difference $W-s\Id$ rather than $W$ itself, or equivalently to study the operator $R-dsE$ (recall that $\Id^R=dE$). The following theorem (proved in Section \ref{sec:proof-moments}) describes the asymptotic behaviour of the singular values of $R-dsE$. We denote by $\mathrm{Cat}_p := \frac{1}{p+1} \binom{2p}{p}$ the $p$th Catalan number.

Recall that the (standard) quarter-circle distribution is the probability measure
\[\mu_{qc} = \frac{\sqrt{4 - x^2}}{\pi} 1_{[0,2]}(x) \, dx.\]
The even moments of this measure are given by the Catalan numbers: 
\[ \int_0^{2} x^{2p} d\mu_{qc}(x) = \mathrm{Cat}_p , \]
whereas the odd moments are
\[ \int_0^{2} x^{2p+1} d\mu_{qc}(x) =  \frac{2^{4p+5}p!(p+2)!}{\pi (2p+4)!}. \]
In particular, the average of $\mu_{qc}$ is $8/(3\pi)$.

\begin{theorem} \label{theo:quartercircle}
For every integers $s,d$, let $R_{d^2,s}$ be the realignment of a Wishart matrix $W_{d^2,s}$, and $Q=Q_{d^2,s}=\frac{1}{d\sqrt{s}}(R_{d^2,s}-dsE)$. Then, when $d$ and $s$ tend to infinity, 
 \[ \lim_{d,s \to \iy} \E \frac{1}{d^2} \tr [(QQ^*)^p] = \mathrm{Cat}_p, \]
\[ \lim_{d,s \to \iy} \Var \frac{1}{d^2} \tr [(QQ^*)^p] = 0. \]
\end{theorem}

We emphasize that there is no assumption about the relative growth of $s$ and $d$, besides the fact that they both tend to infinity.

Theorem \ref{theo:quartercircle} can be restated as follows: the empirical singular value distribution of $Q$, defined as $\frac{1}{d^2} \sum_{i=1}^{d^2} \sigma_i(Q)$, converges in moments towards a quarter-circle distribution $\mu_{qc}$. Via standard arguments (see e.g. \cite{aub} for a sketch), it follows that for every continuous function $f:\R^+ \to \R$ with polynomial growth, we have (convergence in probability)
\begin{equation} \label{eq:test-functions} \frac{1}{d^2} \tr f(|Q_{d^2,s}|) \stackrel{\P}{\longrightarrow} \int_0^2 f d\mu_{sc} .\end{equation}

When applied to the function $f(x)=x$, equation \eqref{eq:test-functions} yields
\begin{equation} \label{eq:equivalent-for-Q} \|Q_{d^2,s}\|_1 \stackrel{\P}{\sim} \frac{8d^2}{3\pi}\end{equation}
(by $X \stackrel{\P}{\sim} Y$ we mean that the ratio $X/Y$ converges in probability towards $1$).

\section{From Wishart matrices to random states: proof of Theorem \ref{thm:threshold-realignment}}

\label{sec:from-wishart-to-states}

In this section we show how to derive Theorem \ref{thm:threshold-realignment} from the results on Wishart matrices. 
Induced random states are closely related to Wishart matrices. Indeed, it is well known that if $W=W_{d^2,s}$ is a Wishart matrix, then $\rho := (\tr W)^{-1} W$ is a random state with distribution $\mu_{d^2,s}$. 

We are going to prove first a weak form of the theorem, where we show that the threshold is of order $d^2$, without identifying the exact constant $\gamma=(8/3\pi)^2$. Here is the relevant proposition; the proof will use only the expansion of $\tr (QQ^*)^p$ for $p=1$ and $p=2$, and will have the advantage to work for all values of $s$ (including $s \ll d^2$ or $s \gg d^2$, which require special attention). 

\begin{proposition} \label{prop:unsharp-threshold}
There exist absolute constants $c,C,c_0$ and $C_0$ such that, for a random state $\rho$ on $\C^d \otimes \C^d$ with distribution $\mu_{d^2,s}$, the following holds
\begin{enumerate}
 \item If $s \leq c_0d^2$, then $\P(\|\rho^R\|_1 > 1) \geq 1-C \exp (-cd^{1/4})$.
 \item If $s \geq C_0d^2$, then $\P(\|\rho^R\|_1 \leq 1) \geq 1- C \exp(-cd^{1/4})$.
\end{enumerate}
\end{proposition}

Assume that $\rho$ is defined by the equation $\rho := (\tr W)^{-1} W$, where $W$ is a Wishart random matrix with parameters $(d^2,s)$. Define $\alpha$ by $\tr W = (1+\alpha)d^2s$. The matrix $Q$ from Theorem \ref{theo:quartercircle} is related to $\rho$ by the equation $\frac{1}{d \sqrt{s}} Q = (1+\alpha) \rho^R-E/d$. In particular,

\begin{equation} \label{eq:from-rho-to-Q}
\frac{1}{1+\alpha} \left( \frac{\|Q\|_1}{d \sqrt{s}} - \frac{1}{d} \right) \leq  \|\rho^R\|_1 \leq 
\frac{1}{1+\alpha} \left( \frac{\|Q\|_1}{d \sqrt{s}} + \frac{1}{d} \right)
\end{equation}
 
The random variable $\tr W$ follows a $\chi^2$ distribution, and  the next lemma implies that with large
probability, $|\alpha| \leq 1/(d\sqrt{s})$.

\begin{lemma}[see e.g. \cite{ml}, Lemma 1] \label{lemma:chi2}
If $W=W_{d^2,s}$ is a Wishart matrix, then for every $0 < \e < 1$,
\[ \P \left( \left| \frac{\tr W}{d^2s} -1 \right| > \e \right) \leq 2 \exp (-c \e^2 d^2 s). \]
\end{lemma}

\begin{proof}[Proof of Proposition \ref{prop:unsharp-threshold}]
It is possible to estimate $\|Q\|_1$ from the knowledge of $\|Q\|_2$ and $\|Q\|_4$, using the following inequalities

\begin{equation} \label{eq:holder} \frac{\|Q\|_2^3}{\|Q\|^2_4} \leq \|Q\|_1 \leq d \|Q\|_2 .\end{equation}

Since $\|Q\|_2^2 = \tr (QQ^*) $ and $\|Q\|_4^4 = \tr (QQ^*QQ^*)$ are polynomials in the matrix entries, they are easier to analyze. Here is a proposition which can be derived from the analysis in Section \ref{sec:proof-moments}---more precisely see equations \eqref{eq:p-equals-1}, \eqref{eq:p-equals-2} and \eqref{eq:upper-bound-variance}.

\begin{proposition}
There are absolute constants $c,C$ such that the following inequalities hold for every $d$ and $s$
\[ cd^2 \leq \E \|Q\|_2^2 \leq Cd^2, \quad \quad \quad cd^2 \leq \E \|Q\|_4^4 \leq Cd^2,\]
\[ \Var \|Q\|_2^2 \leq Cd^2, \quad \quad \quad \Var \|Q\|_4^4 \leq Cd^2 \]
\end{proposition}

We now use a general concentration inequality for polynomials in Gaussian variables.

\begin{proposition}[see \cite{janson}, Theorem 6.7]
Let $(G_i)$ be independent Gaussian variables, and $P$ be a polynomial of (total) degree $q$. Consider the random variable $Y=P(G_1,\dots,G_n)$. Then for every $t>0$,
\[ \P( |Y-\E Y| \geq t \sqrt{\Var Y} ) \leq C_q \exp (-c_q t^{2/q})  \]
($C_q$ and $c_q$ being constants depending only on $q$).
\end{proposition}

Applied to the polynomials $\|Q\|_2^2$ and $\|Q\|_4^4$ (of degree respectively $4$ and $8$), we obtain, that with large probability, both and $\|Q\|_2^2$ and $\|Q\|_4^4$ are of order $d^2$ (up to universal constant). Together with $\eqref{eq:holder}$, this yields that, for some absolute constants $c,C$
\begin{equation} \label{eq:bounds-on-norm} \P( cd^2 \leq \|Q\|_1 \leq Cd^2) \geq 1- C \exp(-cd^{1/4}) \end{equation}

From \eqref{eq:bounds-on-norm} and \eqref{eq:from-rho-to-Q}, we obtain that with probability larger than $1- C \exp(-cd^{1/4})$,
\[ \frac{1}{1+\alpha} \left( \frac{cd}{\sqrt{s}} - \frac{1}{d} \right) \leq  \|\rho^R\|_1 \leq 
\frac{1}{1+\alpha} \left( \frac{Cd}{\sqrt{s}} + \frac{1}{d} \right) \]

and Proposition \ref{prop:unsharp-threshold} follows from Lemma \ref{lemma:chi2}.
\end{proof}

We denote by $\cR$ the (convex) set of states which are not detected to be entangled by the realignment criterion
\[ \cR = \cR(\cH) = \{ \rho \textnormal{ state on $\cH \st \|\rho^R\|_1 \leq 1$} \} . \]
We also introduce the gauge $\|\cdot\|_\cR$ associated to the convex body $\cR$. It is defined for any state $\rho$ as
\begin{eqnarray*} \|\rho\|_{\cR} & = & \inf \left\{ t \geq 0 \st \frac{\Id}{d^2} + \frac{1}{t} \left(\rho-\frac{\Id}{d^2}\right) \in \cR \right\} \\
 & = & \inf \left\{ t \geq 0 \st \left\| \frac{E}{d} + \frac{1}{t} \left(\rho^R-\frac{E}{d}\right) \right\|_1 \leq 1
 \right\}.
\end{eqnarray*}

\begin{lemma} \label{lemma:norms} The following inequalities hold for any state $\rho$ on $\C^d \otimes \C^d$,
\[ \frac{d}{d+1} \| \rho^R - E/d \|_1 \leq \|\rho\|_\cR \leq \frac{d}{d-1} \| \rho^R - E/d \|_1 .\]
\end{lemma}

\begin{proof}
Let $\lambda = \|\rho\|_\cR$. Then $\left\| \frac{E}{d} + \frac{1}{\lambda} \left(\rho^R-\frac{E}{d}\right) \right\|_1=1$, and by the triangle inequality
\[ 1-\frac{1}{d} \leq \frac{1}{\lambda}  \left\| \rho^R-\frac{E}{d} \right\|_1 \leq 1+ \frac{1}{d} \]
and the result follows. 
\end{proof}

We are going to use a result from \cite{asy}, which gives a concentration inequality for the gauge of a random state (note that the inradius of $\cR$ equals the inradius of the set of separable states, which is $1/\sqrt{d^2(d^2-1)}$, see \cite{gb}). We obtain

\begin{proposition}[\cite{asy}, Proposition 4.2] \label{prop:concentration}
For every constant $c_0>0$, there are constants $c,C$ such that the following hold. Let $s \geq c_0d^2$, and $\rho$ be a random state on $\C^d \otimes \C^d$ with distribution $\mu_{d^2,s}$, then
\[ \P \left( \left| \|\rho\|_\cR - M \right|  \geq \eta \right) \leq C  \exp(-cs) + C\exp(-cs \eta^2) ,\]
where $M=M_{d^2,s}$ is the median of the random variable $\|\rho\|_\cR$. 
\end{proposition}

Note that Proposition 4.2 in \cite{asy} is stated with the restriction $s \geq d^2$; however it can be checked that the proof extends to the range $s \geq c_0d^2$ for any $c_0>0$, altering only the values of $c$ and $C$. We can now pass to the complete proof of Theorem \ref{thm:threshold-realignment},

\begin{proof}[Proof of Theorem \ref{thm:threshold-realignment}]

Let us show the first part of the statement, the second part being similar. Since the case when $s \leq c_0d^2$ was covered by Proposition \ref{prop:unsharp-threshold}, we may assume $s \geq c_0d^2$. 

Denote by $\pi_{d^2,s}$ the probability that a random state on $\C^d \otimes \C^d$, with distribution $\mu_{d^2,s}$,
belongs to the set $\cR$. Fix $\e > 0$. For each $d$, let $s=s_d$ be the number with $c_0d^2 \leq s \leq (1-\e)\gamma d^2$ such that $\pi_{d^2,s}$ is maximal. We claim that
\[ \liminf_{d \to \iy} M_{d^2,s_d} \geq \frac{1}{\sqrt{1-\e}} .\]

Indeed, by Lemma \ref{lemma:norms}, the random variables $\|\rho\|_\cR$ and $\|\rho^R-E/D\|_1$ have asymptotically equivalent medians. Moreover, by \eqref{eq:from-rho-to-Q}, it suffices to show that
\[ \liminf_{d \to \iy} \ \textnormal{Median} \left( \frac{\|Q\|_1}{d \sqrt{s}} \right) \geq \frac{1}{\sqrt{1-\e}} ,\]
and this last inequality follows immediately from \eqref{eq:equivalent-for-Q}.

Choose some $\eta$ such that $0<\eta < \frac{1}{\sqrt{1-\e}}-1$. Applying Proposition \ref{prop:concentration}, we obtain, for $d$ large enough,
\[ \pi_{d^2,s} = \P ( \rho \in \cR) \leq \P \left( \|\rho\|_\cR < M_{d^2,s_d} - \eta \right) \leq C(\e) \exp(-c(\e)s_d).\]
Small values of $d$ are taken into account by adjusting the constants if necessary. 

For the second part of the theorem, consider the number $s=s'_d$ with $s \geq (1-\e)\gamma d^2$ such that $\pi_{d^2,s}$ is minimal. This number is well-defined since for fixed $d$, the sequence $\pi_{d^2,s}$ converges to $1$ as $s$ tends to infinity (the measures $\mu_{d^2,s}$ converge towards the Dirac mass at the maximally mixed state). The rest of the proof is similar.
\end{proof}

\section{Background on combinatorics of non-crossing partitions and graphical calculus}

\label{sec:graphical-calculus}

Let us first recall a number of results from the combinatorial theory of noncrossing partitions; see \cite{nsp} for a detailed presentation of the theory. For a permutation $\sigma \in S_p$, we introduce the following standard notation:
\begin{itemize}
\item $\# \sigma$ is the number of cycles of $\sigma$;
\item $|\sigma|$ is its length, defined as the minimal number $k$ such that $\sigma$ can be written as a product of $k$ transpositions. The function $(\sigma, \pi) \to |\sigma^{-1}\pi|$ defines a distance on $S_p$. One has $\#\sigma + |\sigma| = p$.
\end{itemize}
Let $\xi\in S_p$ be the canonical full cycle $\xi = (1 2 \cdots p)$. The set of permutations $\sigma \in S_p$ which saturate the triangular inequality $|\sigma| + |\sigma^{-1} \xi| = |\xi| = p-1$ is in bijection with the set $NC(p)$ of noncrossing partitions of $[p]:=\{1,\ldots,p\}$. We call such permutations \emph{geodesic} and we shall not distinguish between a non crossing partition and its associated geodesic permutation. We also recall a well known bijection between $NC(p)$ and the set $NC_2(2p)$ of noncrossing \emph{pairings} of $2p$ elements. To a noncrossing partition $\pi \in NC(p)$ we associate an element $\mathrm{fat}(\pi) \in NC_2(2p)$ as follows: for each block $\{i_1, i_2, \ldots, i_k\}$ of $\pi$, we add the pairings $\{2i_1-1, 2i_k\}, \{2i_1, 2i_2-1\}, \{2i_2, 2i_3-1\}, \ldots, \{2i_{k-1}, 2i_k-1\}$ to $\mathrm{fat}(\pi)$. The inverse operation is given by collapsing the elements $2i-1, 2i \in \{1, \ldots, 2p\}$ to a single element $i \in \{1, \ldots, p\}$.

In the rest of the paper we shall perform moment computations for random matrices with Gaussian entries. The main tool we use is the Wick formula (see e.g.  \cite{zvonkin} for a proof)

\begin{lemma}\label{lem:wick}
Let $X_1,\dots,X_k$ be jointly Gaussian centered random variables.
If $k=2p+1$ then $\E[X_1\cdots X_k]=0$. If $k=2p$ then
\begin{equation}\label{eq:wick}
\E[X_1\cdots X_k]=\sum_{\substack{\pi=\{\{i_1,j_1\},\ldots , \{i_p,j_p\}\} \\ \text{pairing of }\{1,\ldots ,k\} }} \quad \prod_{m=1}^p \E[X_{i_m}X_{j_m}]
\end{equation}
\end{lemma}

The above formula is very useful when applied to moments of Gaussian random matrices. Moreover, a \emph{graphical formalism} adapted to random matrices was developed in \cite{cn2} in order to facilitate its application. This graphical calculus is similar to the one introduced in \cite{cn1} for unitary integrals and the corresponding Weingarten formula. We present next the basic ideas of the Gaussian graphical calculus and we refer the interested reader to \cite{cn2} for the details. 

In the Gaussian graphical calculus tensors (and, in particular, matrices) are represented by \emph{boxes}. In order to specify the vector space a tensor belongs to, boxes are decorated with differently shaped symbols, where each symbol corresponds to a vector space. The symbols are empty (white) or filled (black), corresponding to primal or dual spaces. Tensor contractions are represented graphically by \emph{wires} connecting these symbols. A wire should always connect two symbols of the same shape (corresponding thus to vector spaces of the same dimension). A wire connecting an empty symbol with a filled symbol of the same shape corresponds to the canonical map $\C^n \otimes (\C^n)^* \to \C$. However, we shall allow wires connecting two white or black symbols, by identifying non-isomorphically a vector space with its dual. Finally, a \emph{diagram} is a collection of such boxes and wires and corresponds to an element in a tensor product space.  

The main advantage of such a representation is that it provides an efficient way of computing expectation values of such tensors when some (or all) of the boxes are random matrices with i.i.d. Gaussian entries. Indeed, there exists an efficient way of implementing the Wick formula in Lemma \ref{lem:wick}. When the entries of the Gaussian matrices have standard normal distributions, the covariances in equation \eqref{eq:wick} are just delta functions. We state now the graphical Wick formula from \cite{cn2}. 

\begin{theorem}\label{thm:Wick_diag}
Let $\D$ a diagram containing $p$ boxes $X$ and $p$ boxes $\bar X$ which correspond to random matrices with i.i.d. standard Gaussian entries. Then
\[\E_X[\D]=\sum_{\alpha \in \mathcal S_p} \D_\alpha,\]
where the diagram $\D_\alpha$ is constructed as follows. One starts by erasing the boxes $X$ and $\bar X$, but keeps the symbols attached to these boxes. Then, the decorations (white and black) of the $i$-th $X$ box are paired with the decorations of the $\alpha(i)$-th $\bar X$ box in a coherent manner, see Figure \ref{fig:expectation_Gaussian}. In this way, we obtain a new diagram $\D_\alpha$ which does not contain any $X$ or $\bar X$ boxes. The resulting diagrams $\mathcal D_\alpha$ may contain loops, which correspond to scalars; these scalars are equal to the dimension of the vector space associated to the decorations.
\end{theorem}

\begin{figure}
\includegraphics{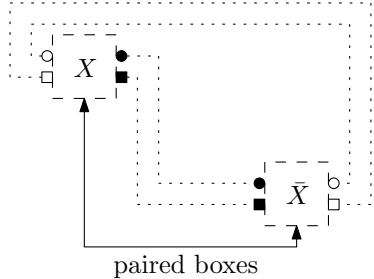}
\caption{Pairing of boxes in the Gaussian graphical calculus}
\label{fig:expectation_Gaussian}
\end{figure}

\begin{figure}
\includegraphics{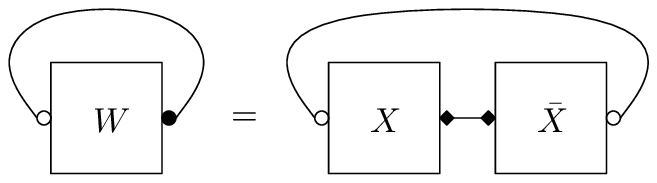}\\
\vspace{.5cm}
\includegraphics{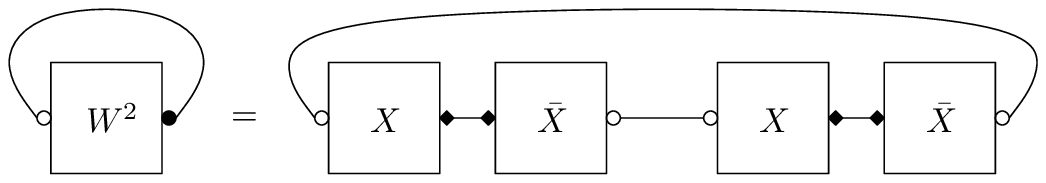}
\caption{Diagrams for the first two moments of a Wishart matrix. Round symbols correspond to $\C^d$ and diamond symbols correspond to $\C^s$.}
\label{fig:W-moments}
\end{figure}

We now present a simple example of moment computation that makes use of the Gaussian graphical calculus. Let $W \in M_{d}(\C)$ be a complex Wishart matrix of parameters $(d, s)$, that is $W = XX^*$ with $X \in M_{d \times s}(\C)$ with i.i.d.~ standard complex Gaussian entries. The diagrams for the first and the second moment of $W$ are presented in Figure \ref{fig:W-moments}. Since these diagrams contain only Gaussian boxes, the resulting expanded diagrams $\mathcal D_\alpha$ will contain only loops, so they will be scalars. For the first moment, the diagram contains only one pair of Gaussian matrices $X / \bar X$, hence the expected value of the trace is given by the following one term sum (see Figure \ref{fig:TrW})
$$\E \tr W = \sum_{\alpha \in \mathcal S_1} \mathcal D_\alpha = ds.$$

\begin{figure}
\includegraphics{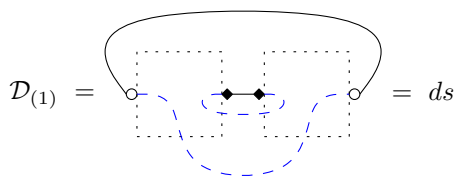}
\caption{Graphical expansion for the first moment of a Wishart matrix. There is only one term in the sum, corresponding to the unique permutation on one element.}
\label{fig:TrW}
\end{figure}

\begin{figure}
\includegraphics{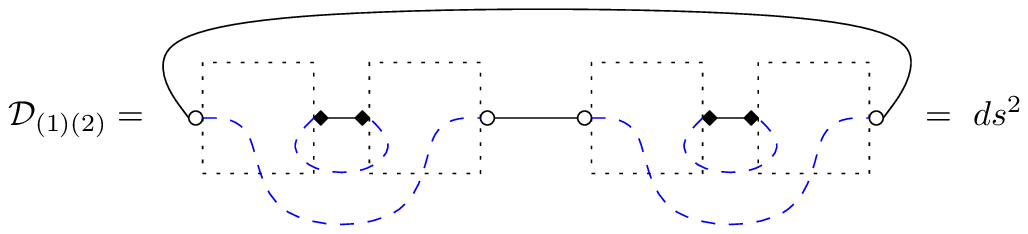}\\
\vspace{.5cm}
\includegraphics{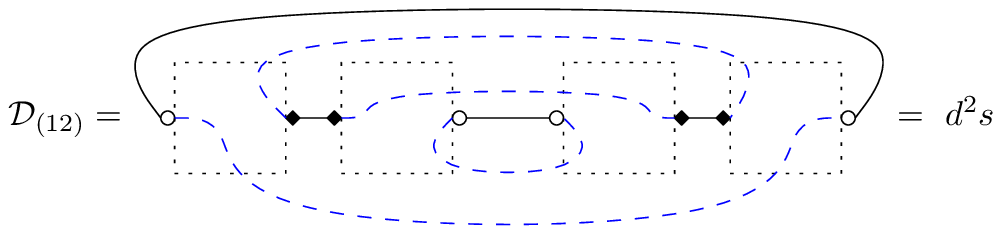}
\caption{Graphical expansion for the second moment of a Wishart matrix. There are two terms in the sum, corresponding to the identity permutation $(1)(2)$ and to the transposition $(12)$.}
\label{fig:TrW2}
\end{figure}

For the second moment, there are two pairs of Gaussian boxes, thus the formula in Theorem \ref{thm:Wick_diag} contains two terms, see Figure \ref{fig:TrW2}:
$$\E \tr W^2 = \sum_{\alpha \in \mathcal S_2} \mathcal D_\alpha = \mathcal D_{(1)(2)} + \mathcal D_{(12)} = ds^2 + d^2s.$$

\section{Moment formula for the singular values of a realigned Wishart matrix}

\label{sec:moments-wishart}

In this section we deduce a formula for realigned Wishart matrices, using Theorem \ref{thm:Wick_diag}. We are going to work in the more general setting of unbalanced tensor products. We consider Wishart matrices  $W \in M_{d_1}(\C) \otimes M_{d_2}(\C)$ of parameters $(d_1d_2, s)$, i.e.~ $W = XX^*$ with $X \in M_{d_1d_2 \times s}(\C)$ having i.i.d.~ $\mathcal N_\C(0,1)$ entries. 

Let $R = W^R \in M_{d_1^2 \times d_2^2}(\C)$ the realigned version of $W$, that is
$$R_{ij,kl} = W_{ik,jl}.$$ 

The diagram of the matrix $R$ is presented in Figure \ref{fig:R}.

\begin{figure}
\includegraphics{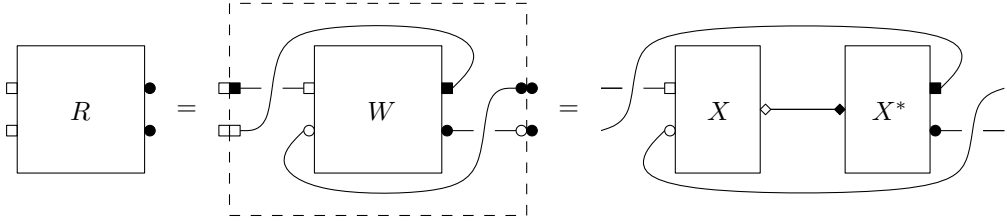}
\caption{The diagram of a realigned Wishart matrix. Square symbols correspond to $\C^{d_1}$, round symbols correspond to $\C^{d_2}$ and diamond-shaped labels correspond to $\C^s$.}
\label{fig:R}
\end{figure}

\begin{proposition}\label{prop:moments-RR*}
The moments of the random matrix $RR^*$ are given by
\begin{equation}\label{eq:moments-RR*}
\E \mathrm{Tr} \left[ (RR^*)^p \right] = \sum_{\alpha \in \mathcal S_{2p}} s^{\#\alpha} d_2^{\#(\alpha\gamma^{-1})} d_1^{\#(\alpha\delta^{-1})},
\end{equation}
where the permutations $\gamma, \delta \in \mathcal S_{2p}$ are given by
\begin{align*}
	\gamma &= (12)(34) \cdots (2p-1, 2p) \\
	\delta &= (1, 2p)(23)(45) \cdots (2p-2, 2p-1).
\end{align*}
\end{proposition}

\begin{figure}
\includegraphics{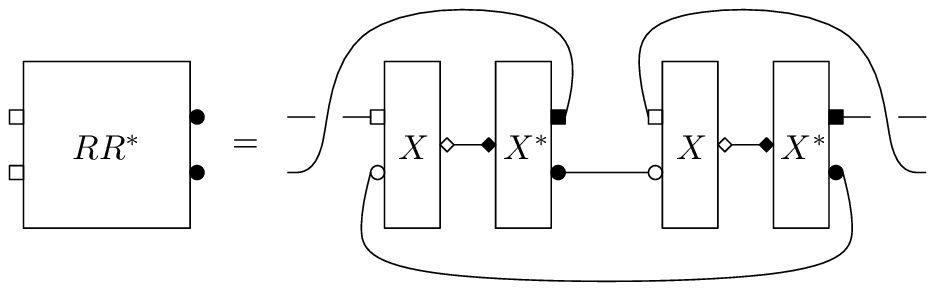}\\
\vspace{.5cm}
\includegraphics{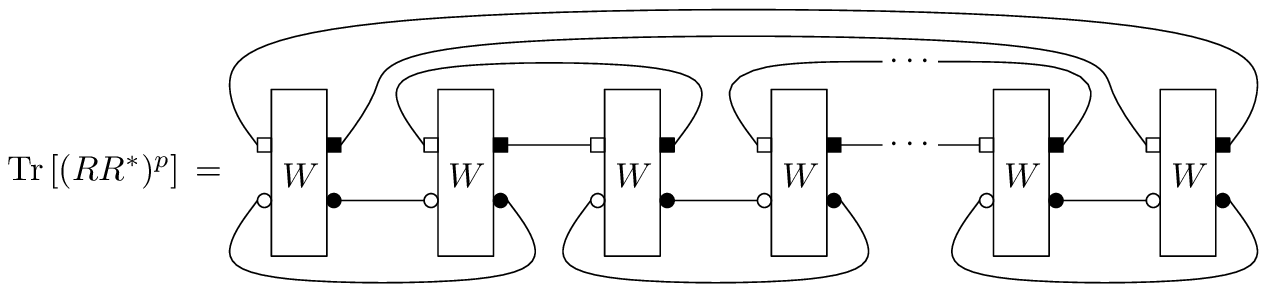}
\caption{The diagrams of the product $RR^*$ and of the $p$-th moment of $RR^*$.}
\label{fig:RRstar}
\end{figure}

\begin{proof}
In Figure \ref{fig:RRstar} we represent the matrix $RR^*$ and its $p$-th moment. The second diagram contains $2p$ $W$ boxes, each $W$ box being represented in Figure \ref{fig:R}. We use Theorem \ref{thm:Wick_diag} to compute the average
$$\E \mathrm{Tr} \left[ (RR^*)^p \right] = \sum_{\alpha \in \mathcal S_{2p}} \mathcal D_\alpha,$$
where $D_\alpha$ is the diagram obtained after the removal of the $X$ and $X^*$ boxes and connecting the $i$-th $X$ box with the $\alpha(i)$-th $X^*$ box. Since the random boxes are the only tensors appearing in the diagram, each $D_\alpha$ will contain only loops that can be counted in the following way:
\begin{enumerate}
	\item There are $\#\alpha$ loops coming from squared labels (associated to $\C^s$), because the initial wiring
	of these labels is given by the identity permutation.
	\item There are $\#(\alpha\gamma^{-1})$ loops coming from the lower round labels (associated to $\C^{d_2}$). The initial wiring of these labels is given by the permutation
	$$\gamma = (12)(34) \cdots (2p-1, 2p).$$
	\item There are $\#(\alpha\delta^{-1})$ loops coming from the upper square labels (associated to $\C^{d_1}$). The initial wiring of these labels is given by the permutation
	$$\delta = (1, 2p)(23)(45) \cdots (2p-2, 2p-1).$$
\end{enumerate}
Taking into account all contributions, we obtain the announced moment formula for $RR^*$.
\end{proof}

\section{Proof of Theorem \ref{theo:quartercircle}} \label{sec:proof-moments}

Recall that in Theorem \ref{theo:quartercircle} we are considering the balanced case, $d_1=d_2=d$. Note that
$$QQ^* = d^{-2}s^{-1}(RR^* - dsRE_d - dsE_dR^* + d^2s^2E_d),$$
so that one can expand the $p$-th moment as
$$\E \mathrm{Tr} \left[ (QQ^*)^p \right] = d^{-2p}s^{-p} \sum_{f_{1,2}:[p] \to \{0,1\}} (-1)^{\sum_i (f_1(i)+f_2(i))} \E \mathrm{Tr} \prod_{i=1}^p \left( F_1(i)F_2(i) \right),$$
where 
$$F_1(i) = \begin{cases}
R \quad &\text{ when } \quad f_1(i)=0,\\
dsE_d \quad &\text{ when } \quad f_1(i)=1,
\end{cases}$$
and 
$$F_2(i) = \begin{cases}
R^* \quad &\text{ when } \quad f_2(i)=0,\\
dsE_d \quad &\text{ when } \quad f_2(i)=1.
\end{cases}$$
We are going to use now a trick that will allow us to compute the expected value in the general term above in the same manner as we did for $RR^*$ in Proposition \ref{prop:moments-RR*}. The idea, presented graphically in Figure \ref{fig:dsREd}, is that when one uses the graphical expansion formula for the expected value above, it is as if we had only $RR^*$ terms, but the set of permutations we allow is restricted to 
$$\mathcal S_{2p}(f_1,f_2) = \{\alpha \in \mathcal S_{2p} \, | \, \forall i \in f_1^{-1}(1), \, \alpha(2i-1) = 2i-1 \text{ and } \forall i \in f_2^{-1}(1), \, \alpha(2i) = 2i\}.$$
\begin{figure}
\includegraphics{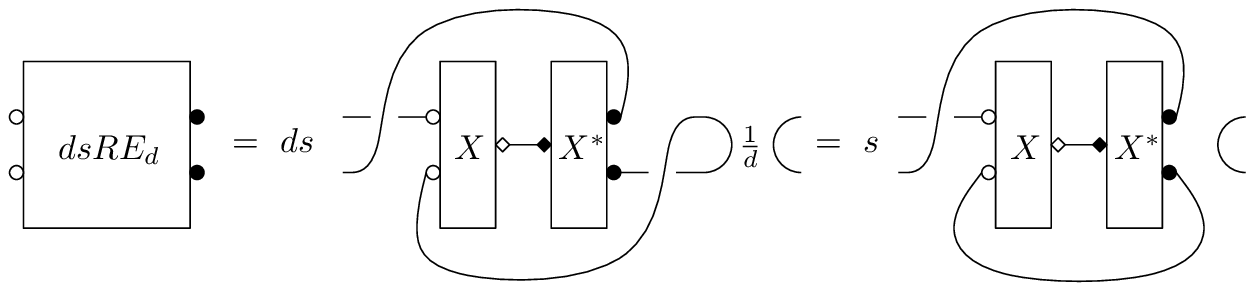}\\
\vspace{.5cm}
\includegraphics{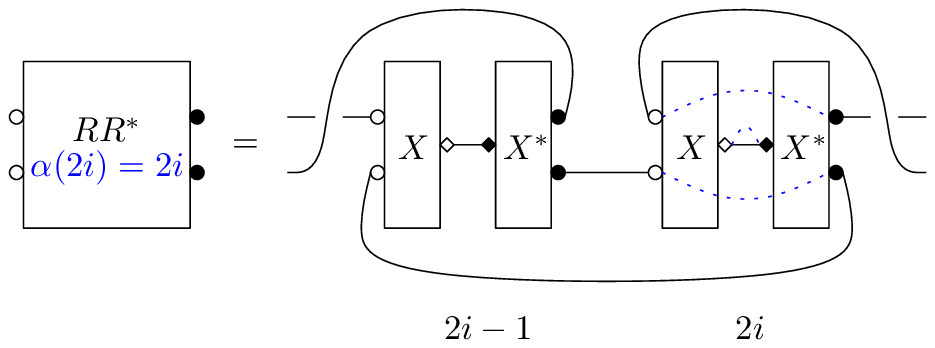}
\caption{The diagrams of $dsRE_d$ and of $RR^*$, with the constraint $\alpha(2i)=2i$ are the same.}
\label{fig:dsREd}
\end{figure}
Using the result in Proposition \ref{prop:moments-RR*}, we obtain
\begin{align*}
\E \mathrm{Tr} \left[ (QQ^*)^p \right] &= d^{-2p}s^{-p} \sum_{f_{1,2}:[p] \to \{0,1\}} (-1)^{\sum_i (f_1(i)+f_2(i))} \sum_{\alpha \in \mathcal S_{2p}(f_1,f_2)} s^{\#\alpha} d^{\#(\alpha\gamma^{-1})} d^{\#(\alpha\delta^{-1})}\\
&= d^{-2p}s^{-p} \sum_{\alpha \in \mathcal S_{2p}} s^{\#\alpha} d^{\#(\alpha\gamma^{-1})} d^{\#(\alpha\delta^{-1})}\sum_{(f_1,f_2) \in \mathcal F(\alpha)} (-1)^{\sum_i (f_1(i)+f_2(i))}. 
\end{align*}
Note that in the last equality we inverted the summation order, so we had to restrict the set of choice functions $f_{1,2}$ to the family
$$\mathcal F(\alpha) = \{ (f_1,f_2) \, | \, \alpha(2i-1)\neq 2i-1 \implies f_1(i)=0 \text{ and }  \alpha(2i)\neq 2i \implies f_2(i)=0 \}.$$
Let us show now that whenever a permutation $\alpha \in \mathcal S_{2p}$ has a fixed point $i_0$, the corresponding sum 
\begin{equation}\label{eq:sum-f12}
\sum_{(f_1,f_2) \in \mathcal F(\alpha)} (-1)^{\sum_i (f_1(i)+f_2(i))}
\end{equation}
equals zero. For such a permutation and an element $f=(f_1,f_2) \in \mathcal F(\alpha)$, define another pair $\tilde f = (\tilde f_1, \tilde f_2)$ as follows. If $i_0$ is odd, $i_0 = 2j_0-1$, then put $\tilde f_2 = f_2$ and 
$$\tilde f_1(j) = \begin{cases}
1-f_1(j) \quad &\text{ if } \quad j=j_0,\\
f_1(j) \quad &\text{ if } \quad j \neq j_0.
\end{cases}$$
For even $i_0 = 2j_0$, define $\tilde f_1 = f_1$ and 
$$\tilde f_2(j) = \begin{cases}
1-f_2(j) \quad &\text{ if } \quad j=j_0,\\
f_2(j) \quad &\text{ if } \quad j \neq j_0.
\end{cases}$$
Since $i_0$ is a fixed point of $\alpha$, we have $\tilde f = (\tilde f_1, \tilde f_2) \in \mathcal F(\alpha)$ and the map $f \to \tilde f$ is thus an involution without fixed points acting on $F(\alpha)$. Notice also that changing a single value in a pair $(f_1,f_2)$ changes the parity of the sum $\sum_i (f_1(i)+f_2(i))$. This concludes our argument that the sum \eqref{eq:sum-f12} is null whenever $\alpha$ has a fixed point.

We have thus shown that permutations with fixed points cancel each other out, so we have
$$\E \mathrm{Tr} \left[ (QQ^*)^p \right] =  d^{-2p}s^{-p} \sum_{\alpha \in \mathcal S_{2p}^o} s^{\#\alpha} d^{\#(\alpha\gamma^{-1})} d^{\#(\alpha\delta^{-1})}\sum_{(f_1,f_2) \in \mathcal F(\alpha)} (-1)^{\sum_i (f_1(i)+f_2(i))},$$
where we denote by $\mathcal S_{2p}^o$ the set of permutations of $[2p]$ without fixed points. For such a permutation $\alpha$, the set of admissible choices $\mathcal F(\alpha)$ contains only one element $f=(f_1,f_2)$ with $f_1(i)=f_2(i)=0$ in such a way that the above formula simplifies to
\begin{equation}\label{eq:sum-QQ*}
\E \mathrm{Tr} \left[ (QQ^*)^p \right] =  \sum_{\alpha \in \mathcal S_{2p}^o} d^{-2p}s^{-p} s^{\#\alpha} d^{\#(\alpha\gamma^{-1})} d^{\#(\alpha\delta^{-1})}.
\end{equation}

For small values of $p$, we obtain
\begin{equation} \label{eq:p-equals-1} \E \mathrm{Tr} \left[ QQ^* \right] = d^2 ,\end{equation}
\begin{equation} \label{eq:p-equals-2} \E \mathrm{Tr} \left[ (QQ^*)^2 \right] = 2d^2 + 2s^{-1}d^2 + 1 + 4s^{-1} .
\end{equation}

We show next that the dominating term in the sum above is of the order $d^2$ and that it is given by permutations $\alpha$ which are non-crossing pair partitions of $[2p]$. Since there are $\mathrm{Cat}_p$ such permutations, we obtain
\begin{equation}\label{eq:moment-QQ*}
\E \mathrm{Tr} \left[ (QQ^*)^p \right] =  d^2 \mathrm{Cat}_p (1+o(1)),
\end{equation}
which is the moment formula we aimed for.

Let us consider separately the exponents of $d$ and $s$ in the general term of the sum \eqref{eq:sum-QQ*}, 
\begin{align*}
g(\alpha) &= -2p+\#(\alpha\gamma^{-1}) + \#(\alpha\delta^{-1}) = 2p - (|\alpha\gamma^{-1}| + |\alpha\delta^{-1}|)\\
h(\alpha) &= -p + \#\alpha = p-|\alpha|.
\end{align*}
Using the the triangular inequality and the fact that $\alpha$ has no fixed points, we obtain
\begin{align*}
|\alpha\gamma^{-1}| + |\alpha\delta^{-1}| &\geq |\gamma\delta| = |(2p-1 \, 2p-3 \cdots 5 3 1)(2 4 6 \cdots 2p)| = 2p-2,\\
|\alpha| &\geq p,
\end{align*}
which shows that we have indeed $g(\alpha) \leq 2$ and $h(\alpha) \leq 0$. In order to conclude, it remains to be shown that the permutations which saturate both inequalities are exactly the non-crossing pair partitions of $[2p]$. The fact that $\alpha$ has no fixed points and that $|\alpha| = p$ implies that $\alpha$ is indeed a product of $p$ disjoint transpositions i.e. a pair partition. To show that it is non-crossing, we start from the geodesic condition $|\alpha\gamma^{-1}| + |\alpha\delta^{-1}| = |\gamma\delta| = |(2p-1 \, 2p-3 \cdots 5 3 1)(2 4 6 \cdots 2p)|$. This implies that the permutation $\gamma \alpha$ lies on the geodesic $\mathrm{id} \to \gamma \delta$. More precisely, we can write $\alpha = \gamma \Pi_o \Pi_e$ where $\Pi_o$ and $\Pi_e$ are permutations acting on the odd, respectively even elements of $[2p]$. The geodesic condition implies that these permutations come from non-crossing partitions $\pi_o, \pi_e \in NC(p)$:
\begin{align*}
\Pi_o(2i-1) &= 2\pi_o^{-1}(i)-1,\\
\Pi_o(2i) &= 2i,\\
\Pi_e(2i-1) &= 2i-1,\\
\Pi_e(2i) &= 2\pi_e(i).
\end{align*}
The fact that $\alpha$ is an involution easily implies $\pi_e = \pi_o =: \pi \in NC(p)$ so that the action of $\alpha$ is given by
\begin{align*}
\alpha(2i-1) &= 2\pi^{-1}(i),\\
\alpha(2i) &= 2\pi(i)-1.
\end{align*}
This is equivalent to $\alpha = \mathrm{fat}(\pi)$ so that $\alpha$ is necessarily a non-crossing pair partition.

\bigskip

Let us now prove the second statement in the theorem, by giving an estimate on the second moment of the random variable $\mathrm{Tr} \left[ (QQ^*)^p \right]$.

Using Theorem \ref{thm:Wick_diag} for the diagram of $\E \mathrm{Tr}^2 \left[ (QQ^*)^p \right]$, which is made of two disconnected copies of the bottom diagram in Figure \ref{fig:RRstar} we obtain the moment expansion
\begin{align*}
\E \mathrm{Tr}^2 \left[ (QQ^*)^p \right] &= d^{-4p}s^{-2p} \sum_{f_{1,2}:[2p] \to \{0,1\}} (-1)^{\sum_i (f_1(i)+f_2(i))}  \cdot \\
&\qquad \qquad \cdot \E \left[\mathrm{Tr} \left(\prod_{i=1}^p \left( F_1(i)F_2(i) \right)\right) \left(\mathrm{Tr}\prod_{i=p+1}^{2p} \left( F_1(i)F_2(i) \right)\right)\right] \\
&= d^{-4p}s^{-2p} \sum_{\alpha \in \mathcal S_{4p}} s^{\#\alpha} d^{\#(\alpha\gamma_{12}^{-1})} d^{\#(\alpha\delta_{12}^{-1})}\sum_{(f_1,f_2) \in \mathcal F(\alpha)} (-1)^{\sum_i (f_1(i)+f_2(i))}. 
\end{align*}
where the $F$ and $f$ functions have the same meaning as before and the permutations $\gamma_{12},\delta_{12}$ are defined by
\begin{align*}
	\gamma_{12} &= (12)(34) \cdots (2p-1, 2p) (2p+1, 2p+2) \cdots (4p-1, 4p)\\
	\delta_{12} &= (1, 2p)(23)(45) \cdots (2p-2, 2p-1)(2p+1, 4p)(2p+2,2p+3) \cdots (4p-2, 4p-1).
\end{align*}

We can show, by the same technique as before, that permutations $\alpha$ with fixed points cancel each other out in the sum above, so we have
$$\E \mathrm{Tr}^2 \left[ (QQ^*)^p \right] =  d^{-4p}s^{-2p} \sum_{\alpha \in \mathcal S_{4p}^o} s^{\#\alpha} d^{\#(\alpha\gamma_{12}^{-1})} d^{\#(\alpha\delta_{12}^{-1})}
$$

We investigate next the dominating term in the sum above. The exponents of $d$ and $s$ in the general term read
\begin{align*}
g_{12}(\alpha) &= -4p+\#(\alpha\gamma_{12}^{-1}) + \#(\alpha\delta_{12}^{-1}) = 4p - (|\alpha\gamma_{12}^{-1}| + |\alpha\delta_{12}^{-1}|)\\
h_{12}(\alpha) &= -2p + \#\alpha = 2p-|\alpha|.
\end{align*}
Using the the triangular inequality and the fact that $\alpha$ has no fixed points, we obtain
\begin{align*}
|\alpha\gamma_{12}^{-1}| + |\alpha\delta_{12}^{-1}| &\geq |\gamma_{12}\delta_{12}| = 4p-4,\\
|\alpha| &\geq 2p,
\end{align*}
which shows that we have $g_{12}(\alpha) \leq 4$ and $h_{12}(\alpha) \leq 0$. This proves that the sum behaves like $d^4$ times a constant. To evaluate this constant, we need to find the set of permutations $\alpha$ which saturate the above inequalities. As before, this set is made of pair partitions such that $\gamma_{12} \alpha$ lies on the geodesic $\mathrm{id} \to \gamma_{12} \delta_{12}$. Recall that both permutations $\gamma_{12}$ and $\delta_{12}$ have a product structure
\begin{align*}
\gamma_{12} &= \gamma_1 \gamma_2\\
\delta_{12} &= \delta_1 \delta_2,
\end{align*}
where $\gamma_1, \delta_1$ act on $[2p]$ and $\gamma_2,\delta_2$ act on $2p+[2p]$. Hence, mimicking the reasoning in the first part of the proof, the permutations we want are those having also a product structure $\alpha = \alpha_1 \alpha_2$, where $\alpha_1$ and $\alpha_2$ are non-crossing pair partitions of $[2p]$ and $2p+[2p]$ respectively. Since one can choose $\alpha_1$ and $\alpha_2$ independently, we conclude that 
$$\E \mathrm{Tr}^2 \left[ (QQ^*)^p \right] = d^{4}\left( \mathrm{Cat}_p^2 + o(1) \right),$$
which, together with \eqref{eq:moment-QQ*}, achieves the proof of the theorem.

Note that the function $g_{12}(\alpha)$ takes only even values, and therefore permutations $\alpha$ such that $g_{12}(\alpha)<4$ must actually satisfy $g_{12}(\alpha) \leq 2$. This remark yields a bound on the variance
\begin{equation} \label{eq:upper-bound-variance} \Var \tr [(QQ^*)^p] \leq C_pd^{2} ,\end{equation}
where $C_p$ is a constant depending only on $p$.


\section{Realigning states in an unbalanced tensor product}\label{sec:unbalanced}

We analyze now an unbalanced tensor product $\C^{d_1} \otimes \C^{d_2}$, with $d_1 < d_2$. We consider the asymptotic regime where $d_1$ is fixed and $d_2 \to \infty$ and we show that the threshold occurs at a \emph{finite} value of the parameter $s$, more precisely $s=d_1^2$.

\begin{theorem} \label{thm:threshold-realignment-unbalanced}
For every integers $d_1,s$, there are constants $C=C(d_1,s)$ and $c=c(d_1,s)$ such as the following holds. Let $\rho$ be a random state on $\C^{d_1} \otimes \C^{d_2}$ with distribution $\mu_{d_1d_2,s}$, then
\begin{enumerate}
 \item If $s<d_1^2$, then 
 \[ \P( \|\rho^R\| >1) \geq 1-C \exp ( -c d_2^{1/4} ). \] 
 \item If $s > d_1^2$, then
 \[ \P( \|\rho^R\| \leq 1) \geq 1-C \exp ( -c d_2^{1/4}). \] 
\end{enumerate}
\end{theorem}

As in the balanced case, the result is based on a moment computation for a realigned Wishart matrix.

\begin{theorem}\label{thm:unbalanced-finite}
In the regime of fixed $d_1,s$ and $d_2 \to \infty$, the empirical singular value distribution of  $d_2^{-1}R$ converges in moments to a Dirac mass at $\sqrt s$. Moreover, the variances of the moments of the random matrix $d_2^{-1}R$ satisfy
$$\Var \tr \left[ (d_2^{-2}RR^*)^p \right] = O(1/d_2^2).$$
\end{theorem}
\begin{proof}
We need to prove that for every integer $p$,
\[ \lim_{d_2 \to \iy} \frac{1}{d_1^2} \E \tr \left[ (d_2^{-2}RR^*)^p \right] = s^{p/2} .\]

As before, we start from the moment formula \eqref{eq:moments-RR*}
$$\E \mathrm{Tr} \left[ (RR^*)^p \right] = \sum_{\alpha \in \mathcal S_{2p}} s^{\#\alpha} d_2^{\#(\alpha\gamma^{-1})} d_1^{\#(\alpha\delta^{-1})}.$$
Since the only parameter growing to infinity in the above sum is $d_2$, the dominating term is given by the permutation $\alpha=\gamma$ and thus
\begin{equation}\label{eq:dominant-unbalanced}
\E \mathrm{Tr} \left[ (RR^*)^p \right] \sim d_2^{2p} s^p d_1^2,
\end{equation}
showing the convergence in moments of the empirical singular value distribution.

The statement about the variance follows readily from the following formula
$$\E \mathrm{Tr}^2 \left[ (RR^*)^p \right] = \sum_{\alpha \in \mathcal S_{4p}} s^{\#\alpha} d_2^{\#(\alpha\gamma_{12}^{-1})} d_1^{\#(\alpha\delta_{12}^{-1})}.$$
We note that the dominating term is given by $\alpha = \gamma_{12}$ and that it cancels out with the square of the right hand side of equation \eqref{eq:dominant-unbalanced}. The largest remaining terms correspond to permutations $\alpha$ with $\#(\alpha\gamma_{12}^{-1})=4p-2$.
\end{proof}

We now move on to the proof of Theorem \ref{thm:threshold-realignment-unbalanced}, which mimics the one of Proposition \ref{prop:unsharp-threshold}. This approach gives the exact threshold in the unbalanced case because the limiting measure is a Dirac mass, so the bounds given by the moments 2 and 4 are already tight. 

\begin{proof}[Proof of Theorem \ref{thm:threshold-realignment-unbalanced}] 
In the present case, equation \eqref{eq:holder} reads 
$$\frac{\|R\|_2^3}{\|R\|^2_4} \leq \|R\|_1 \leq d_1 \|R\|_2 .$$

Using the previous proposition and the concentration for Gaussian polynomials (lemma \ref{lemma:norms}), we can find constants $c,C$ such that for every $t>0$,
%
\begin{align*}
\P(|\|R\|_2^2- d_2^2d_1^2s(1+o(1))|>Ctd_2)&\leq C \exp(-c t^{1/2})\\
\P(|\|R\|_4^4- d_2^2d_1^2s^2(1+o(1))|>Ctd_2)&\leq C \exp(-c t^{1/4}).
\end{align*}

We choose $t=\eta d_2$ for some $\eta > 0$. The previous two facts imply that with large probability, $\|(d_1d_2s)^{-1}R\|_1$ is close to $d_1/\sqrt{s}$ (the difference being smaller than any fixed $\e >0$, for an appropriate choice of $\eta$). Finally, one can replace $R/(d_1d_2s)$ by $\rho^R$, by using the fact that the trace of the Wishart matrix $W$ concentrates around its mean $d_1d_2s$ (see Lemma \ref{lemma:chi2}).
\end{proof}

\end{document}